\documentclass[a4paper,12pt]{amsart}
\pdfoutput=1
\usepackage{microtype}
\usepackage{amssymb,latexsym}
\usepackage[margin=3.25cm]{geometry}
\usepackage{mathtools}
\usepackage{amsthm}
\usepackage[inline]{enumitem}
\usepackage{mathrsfs}
\usepackage[dvipsnames]{xcolor}
\usepackage{hyperref}
\hypersetup{
    colorlinks=true,
    urlcolor=black,
    linkcolor=black, 
    citecolor=blue} 
\usepackage{mleftright}
 
\DeclareMathOperator{\Span}{span}

\theoremstyle{plain}
\newtheorem{theorem}{Theorem}[section]

\newtheorem{lemma}[theorem]{Lemma}
\newtheorem{corollary}[theorem]{Corollary}
\theoremstyle{remark}
\newtheorem{remark}[theorem]{Remark}
\newtheorem*{remark*}{Remark}
\newtheorem*{remarks*}{Remarks}
\theoremstyle{definition}

\newtheorem{definition}[theorem]{Definition}
\newtheorem*{notation*}{Notation}
\newtheorem*{example*}{Example}

\begin{document}
\title[Total torsion of lines of curvature]{Total torsion of three-dimensional\\ lines of curvature}
\author[M. Raffaelli]{Matteo Raffaelli}
\address{Institute of Discrete Mathematics and Geometry\\
TU Wien\\
Wiedner Hauptstra{\ss}e 8-10/104\\
1040 Vienna\\
Austria}
\email{matteo.raffaelli@tuwien.ac.at}
\thanks{This work was supported by Austrian Science Fund (FWF) project F~77.}
\date{August 24, 2023}
\begin{abstract}
A curve $\gamma$ in a Riemannian manifold $M$ is \textit{three-dimensional} if its torsion (signed second curvature function) is well-defined and all higher-order curvatures vanish identically. In particular, when $\gamma$ lies on an oriented hypersurface $S$ of $M$, we say that $\gamma$ is \textit{well positioned} if the curve's principal normal, its torsion vector, and the surface normal are everywhere coplanar. Suppose that $\gamma$ is three-dimensional and closed. We show that if $\gamma$ is a well-positioned line of curvature of $S$, then its total torsion is an integer multiple of $2\pi$; and that, conversely, if the total torsion of $\gamma$ is an integer multiple of $2\pi$, then there exists an oriented hypersurface of $M$ in which $\gamma$ is a well-positioned line of curvature. Moreover, under the same assumptions, we prove that the total torsion of $\gamma$ vanishes when $S$ is convex. This extends the classical total torsion theorem for spherical curves.
\end{abstract}
\maketitle
\tableofcontents

\section{Introduction and main result}

In classical differential geometry, the \emph{total torsion theorem} states that the total torsion of a closed spherical curve vanishes; see \cite{fenchel1934, santalo1935, scherrer1940, geppert1941} and \cite[p.~170]{millman1977}.

\begin{theorem}\label{TotalTorsionTHM}
Let $I= [0, \ell]$, and let $\gamma \colon I \to \mathbb{R}^{3}$ be a smooth regular curve. If $\gamma$ is closed and $\gamma(I)\in\mathbb{S}^{2}$, then 
\begin{equation*}
\int_{0}^{\ell} \tau\,dt = 0.
\end{equation*}
\end{theorem}

Theorem~\ref{TotalTorsionTHM} manifests the fact that ``the torsion of a closed curve lying on a surface in $\mathbb{R}^{3}$ is somehow constrained by the geometry of [the] surface" \cite[p.~111]{pansonato2008}; see, e.g., \cite{sedykh1994, costa1997, ghomi2017, ghomi2019} for further evidence of the same fact.

Closely related to Theorem~\ref{TotalTorsionTHM} is the following result of Qin and Li.

\begin{theorem}[\cite{qin2002}]\label{surfaceTHM}
Let $S$ be a (smooth) oriented surface in $\mathbb{R}^{3}$. If $\gamma$ is a closed line of curvature of $S$, then the total torsion is an integer multiple of $2\pi$. Conversely, if the total torsion of a closed curve in $\mathbb{R}^{3}$ is an integer multiple of $2\pi$, then it can appear as a line of curvature of an oriented surface.
\end{theorem}

Theorem~\ref{TotalTorsionTHM} and the first part of Theorem~\ref{surfaceTHM} have been generalized to three-dimensional orientable Riemannian manifolds of constant curvature $M^{3}_{c}$ \cite{pansonato2008}; see also \cite{daSilva2020, yin2017} for related results. In the present note we shall see that, under suitable assumptions, both theorems remain valid when $M^{3}_{c}$ is replaced by an arbitrary Riemannian manifold $M^{m}\equiv M$, provided one restricts the attention to \emph{three-dimensional} curves; roughly speaking, a curve in $M$ is three-dimensional if it has one curvature and one ``torsion", all other curvature functions being zero. As we explain below, in that case one should interpret ``torsion" as a signed version of Spivak's ``second curvature function" \cite[p.~22]{spivak1999d}. 

Let $\gamma$ be a unit-speed curve $I \to M$, let $N$ be a unit normal vector field along $\gamma$, and let $\pi_{\mathcal{H}}$ be the orthogonal projection onto $\mathcal{H} = (\gamma' \oplus N)^{\perp}$. We say that $N$ is \textit{torsion-defining} if there exists a smooth unit vector field $W(N)$ along $\gamma$ that is everywhere parallel to $T_{g} = -\pi_{\mathcal{H}}D_{t}N$. If $N$ is torsion-defining, then the function $\tau_{g} = \langle T_{g}, W(N)\rangle$ is called the \textit{(first) geodesic torsion of $\gamma$ with respect to $N$}. In particular, if $D_{t}\gamma'$ is never zero, then the geodesic torsion of $\gamma$ with respect to the principal normal $P = D_{t}\gamma'/\kappa$ is called the \textit{(first) torsion of $\gamma$}, and $\gamma$ is said to be a \textit{Frenet curve}.

The logic behind our terminology is the following. In the same way a generic curve in $\mathbb{R}^{3}$ has one (unsigned) curvature plus one (signed) torsion, a generic curve in $M$ may have one (unsigned) curvature plus $m-2$ (signed) torsions; cf.\ \cite{spivak1999d}. On the other hand, since we never deal with higher-order torsions, we typically speak of ``torsion" as a shorthand for ``first torsion".

Now, to state our generalization of Theorems \ref{TotalTorsionTHM} and \ref{surfaceTHM}, let $S$ be an oriented hypersurface of $M$, and let $N_S$ be its unit normal. A Frenet curve on $S$ is said to be \textit{well positioned} if $N_{S}$, $P$, and $W(P)$ are everywhere coplanar.

\begin{theorem}\label{mainCOR}
Suppose that $\gamma$ is three-dimensional, i.e., that $\gamma$ is a Frenet curve such that $W(P)$ is parallel in $\mathcal{H}(P)$; see Definition~\textup{\ref{3DcurveDEF}}. If $\gamma$ is a well-positioned closed line of curvature of $S$, then the total torsion of $\gamma$ is an integer multiple of $2\pi$; in particular, the total torsion vanishes when $S$ is convex, i.e., when the second fundamental form of $S$ is positive definite. Conversely, if $\gamma$ is open, then there exists an orientable hypersurface in which $\gamma$ is a well-positioned line of curvature; if $\gamma$ is closed, then the same holds provided the total torsion of $\gamma$ is an integer multiple of $2\pi$.
\end{theorem}

Clearly, when $\dim M = 3$, every Frenet curve is three-dimensional, and every Frenet curve on $S$ is well positioned. Specializing the theorem to that case, we may state the following result.

\begin{corollary}
Suppose that $\dim M = 3$ and that $\gamma$ is a closed Frenet curve. If $\gamma$ is a line of curvature of $S$, then the total torsion of $\gamma$ is an integer multiple of $2\pi$; in particular, the total torsion vanishes when $S$ is convex. Conversely, if the total torsion of $\gamma$ is an integer multiple of $2\pi$, then there exists an orientable surface in which $\gamma$ is a line of curvature.
\end{corollary}

\begin{remark}
If $\dim M = 3$, then every regular curve with nonvanishing curvature is a Frenet curve.
\end{remark}

We will obtain Theorem~\ref{mainCOR} as a corollary of a more general statement involving the geodesic torsion of $\gamma$ with respect to an arbitrary unit normal vector field $N$ along $\gamma$, in which the assumption that $\gamma$ is three-dimensional is replaced by the condition that $N$ is a \emph{parallel rotation} of $N_{S}$.

Let $N$ and $Z$ be unit normal vector fields along $\gamma$. We say that $Z$ is a \textit{rotation of $N$} if there exists a continuous unit vector field $H(N, Z) \equiv H$ such that
\begin{enumerate}
\item $\langle H, \gamma' \rangle = \langle H, N \rangle = 0$, i.e., $H \in \Gamma(\mathcal{H})$;
\item $H$, $N$, and $Z$ are everywhere linearly dependent.
\end{enumerate}
Clearly, if $N \wedge Z$ is nowhere zero, then the vector field $H$ is defined up to a sign.

Now, suppose that $Z$ is a rotation of $N$. Then we can write
\begin{equation*}
Z \equiv N(\theta)= - \sin(\theta) H + \cos(\theta) N\
\end{equation*}
for some continuous function $\theta \colon I \to \mathbb{R}$. 

\begin{definition}
A rotation of $N$ is said to be \textit{parallel} if $H$ is parallel with respect to the induced connection on $\mathcal{H}$, and \textit{closed} if $\theta(\ell)-\theta(0) = 2n\pi$ for some $n \in \mathbb{Z}$.
\end{definition}

\begin{remark}
\leavevmode
\begin{enumerate}
\item If $\dim M = 3$, then any unit normal vector field along $\gamma$ is a parallel rotation of $N$.
\item If $\gamma$ is closed, then so is any rotation of $N$.
\end{enumerate}
\end{remark}

\begin{theorem}\label{mainTHM}
If $\gamma$ is a line of curvature of $S$, then the total geodesic torsion of $\gamma$ with respect to any closed parallel rotation of $N_{S}$ is an integer multiple of $2\pi$.
Conversely, suppose that $N$ is torsion-defining and that $W(N)$ is parallel in $\mathcal{H}$. If $\gamma$ is open, then there exists an orientable hypersurface of $M$ in which $\gamma$ is a line of curvature; if $\gamma$ is closed, then the same holds provided the total geodesic torsion of $\gamma$ with respect to $N$ is an integer multiple of $2\pi$.
\end{theorem}

\begin{remark}
It follows from section \ref{RotatingNormal} that if $\gamma$ is a line of curvature of $S$ and $P$ is a parallel rotation of $N_{S}$, then $\gamma$ is three-dimensional.
\end{remark}

The remainder of the paper is organized a follows. In section~\ref{Preliminaries} we set up some notations. In section~\ref{DarbouxCurvatures} we generalize the well-known concepts of geodesic curvature, normal curvature, and geodesic torsion of a curve on a surface in $\mathbb{R}^{3}$ to a curve on a hypersurface of $M$; although, under reasonable assumptions, one may define $m-2$ geodesic curvatures and geodesic torsions, for the sake of simplicity we shall limit ourselves to first-order curvatures. In section~\ref{RotatingNormal} we obtain formulas expressing the curvature vectors of $\gamma$ with respect to a rotation of $N$ in terms of the rotation angle. Finally, in sections \ref{ProofTheorem} and \ref{3DCurves} we prove Theorems \ref{mainTHM} and \ref{mainCOR}, respectively.

\section{Preliminaries}\label{Preliminaries}
In this section we discuss some preliminaries.

Let $M$ be an $m$-dimensional Riemannian manifold, let $\gamma$ be a smooth unit-speed curve $I \to M$, and let $TM \rvert_{\gamma}$ be the ambient tangent bundle over $\gamma$. Recall that
\begin{equation*}
TM\rvert_{\gamma} = \bigsqcup_{t \in  I} T_{\gamma(t)}M.
\end{equation*}
We define a \textit{distribution of rank $r$ along $\gamma$} to be a rank-$r$ subbundle of $TM \rvert_{\gamma}$.

Let $\mathcal{D}$ be a distribution of rank $r$ along $\gamma$, and let $\mathcal{D}^{\perp}$ be the distribution of rank $m-r$ along $\gamma$ whose fiber at $t$ is the orthogonal complement $\mathcal{D}_{t}^{\perp}$ of $\mathcal{D}_{t}$ in $T_{\gamma(t)}M$, so that $TM\lvert_{\gamma}$ splits as
\begin{equation*}
TM\rvert_{\gamma} = \mathcal{D} \oplus  \mathcal{D}^{\perp};
\end{equation*}
accordingly, we write
\begin{equation*}
X = X^{v} + X^{h}
\end{equation*}
for any vector field $X$ along $\gamma$. 

In this setting, the \textit{tangential projection} is the map $\pi_{\mathcal{D}} \colon \Gamma(TM\rvert_{\gamma}) \to \Gamma(\mathcal{D})$ given by
\begin{equation*}
	X \mapsto X^{v}.
\end{equation*}
Likewise, the \textit{normal projection} is the map $\pi^{\perp}_{\mathcal{D}} \colon \Gamma(TM\rvert_{\gamma}) \to \Gamma(\mathcal{D}^{\perp})$ sending each $X$ to the corresponding $X^{h}$.

\section{Darboux curvatures and curvature vectors}\label{DarbouxCurvatures}
The purpose of this section is to extend the classical notions of geodesic curvature, normal curvature, and geodesic torsion of a curve on a surface in $\mathbb{R}^{3}$ to a curve on a hypersurface of $M$.

Let $\gamma$ be a (smooth) unit-speed curve $I \to M$, let $N$ be a unit normal vector field along $\gamma$, and let $\mathcal{H}(N) \equiv \mathcal{H}$ be the distribution of rank $m-2$ along $\gamma$ whose fiber at $t$ is the orthogonal complement of $E(t) = \gamma'(t)$ and $N(t)$ in $T_{\gamma(t)}M$. Denoting by $D_{t}$ the covariant derivative along $\gamma$, we define
\begin{itemize}
\item the \textit{(first) geodesic curvature vector $K_{g}$ of $\gamma$ with respect to $N$} by
\begin{equation*}
K_{g} = \pi_{\mathcal{H}} D_{t}E;
\end{equation*}
\item the \textit{normal curvature vector $K_{n}$ of $\gamma$ with respect to $N$} by 
\begin{equation*}
K_{n} = \pi_{\mathcal{N}} D_{t}E, 
\end{equation*}
where $\mathcal{N} = \Span N$;
\item the \textit{(first) geodesic torsion vector $T_{g}$ of $\gamma$ with respect to $N$} by 
\begin{equation*}
T_{g} = -\pi_{\mathcal{H}} D_{t}N.
\end{equation*}
\end{itemize}

To express these vector fields in coordinates, let $(H_{1}, \dotsc, H_{m-2})$ be a smooth orthonormal frame for $\mathcal{H}$. Then there are functions $\kappa_{g}^{1}, \dotsc, \kappa_{g}^{m-2}$, $\kappa_{n}$, and $\tau_{g}^{1}, \dotsc, \tau_{g}^{m-2}$ such that
\begin{align*}
K_{g} &= \kappa_{g}^{1}H_{1} + \dotsb + \kappa_{g}^{m-2} H_{m-2},\\
K_{n} &= \kappa_{n} N,\\
T_{g} &= \tau_{g}^{1}H_{1} + \dotsb + \tau_{g}^{m-2} H_{m-2}.
\end{align*}

Note that, since $(E, H_{1}, \dotsc, H_{m-2}, N)$ is orthonormal, the following equations hold for all $j =1, \dotsc, m-2$:
\begin{align*}
D_{t}E &= \kappa_{g}^{1}H_{1} + \dotsb + \kappa_{g}^{m-2}H_{m-2} + \kappa_{n}N,\\
D_{t}H_{j} &= -\kappa_{g}^{j}E + \tau_{g}^{j}N + \pi_{\mathcal{H}} D_{t}H_{j},\\
D_{t}N &= -\kappa_{n}E - \tau_{g}^{1}H_{1} - \dotsb - \tau_{g}^{m-2} H_{m-2}.
\end{align*}

The curvature vectors allow us to define corresponding curvature \emph{functions}. In one case, the definition is trivial: the function $\kappa_{n} = \langle D_{t}E, N \rangle$ is called the \textit{normal curvature of $\gamma$ with respect to $N$}. For the remaining two cases, we proceed as follows.

We say that $N$ is \textit{curvature-defining} if there exists a smooth unit vector field $V(N) \equiv V$ along $\gamma$ that is everywhere parallel to $K_{g}$. If $N$ is curvature-defining, then the function $\kappa_{g} = \langle K_{g}, V \rangle$ is called the \textit{(first) geodesic curvature of $\gamma$ with respect to $N$}.

Similarly, we say that $N$ is \textit{torsion-defining} if there exists a smooth unit vector field $W(N) \equiv W$ along $\gamma$ that is everywhere parallel to $T_{g}$. If $N$ is torsion-defining, then the function $\tau_{g} = \langle T_{g}, W \rangle$ is called the \textit{(first) geodesic torsion of $\gamma$ with respect to $N$}.

It is clear that both $\kappa_{g}$ and $\tau_{g}$ are defined up to a sign.

Armed with the notion of geodesic torsion, we may now define torsion. Suppose that the curvature $\kappa = \lVert D_{t}E \rVert$ of $\gamma$ is nowhere zero, so that the \textit{principal normal} $P = D_{t}E/\kappa$ is well-defined. The geodesic torsion vector of $\gamma$ with respect to $P$ is called the \textit{(first) torsion vector of $\gamma$}. In particular, if $P$ is torsion-defining, then the geodesic torsion of $\gamma$ with respect to $P$ is called the \textit{(first) torsion of $\gamma$}.

\begin{remark}
If $P$ is well-defined, then the normal curvature of $\gamma$ with respect to $P$ coincides with the curvature of $\gamma$, while the geodesic curvature with respect to $P$ vanishes.
\end{remark}

To see that our curvature functions naturally extend the classical Darboux curvatures, consider an oriented hypersurface $S$ of $M$, and let $N_{S}$ be its unit normal. If $\gamma$ is a curve on $S$, then the geodesic (resp., normal) curvature vector of $\gamma$ (with respect to $N_{S}$) is the projection onto $TS$ (resp., $NS$) of the ambient acceleration $D_{t}E$ of $\gamma$; and if $\gamma$ is not a geodesic of $M$, then the geodesic torsion vector of $\gamma$ at $\gamma(t)$ is nothing but the torsion vector of the $S$-geodesic passing from $\gamma(t)$ with tangent vector $\gamma'(t)$ \cite[p.~193]{spivak1999c}.

Yet another indication of the naturality of our definition of geodesic torsion is provided by the following lemma, which will play a key role in the proof of Theorem~\ref{mainTHM}.

\begin{lemma}
A curve on $S$ is a line of curvature if and only if its geodesic torsion vector with respect to $N_{S}$ vanishes.
\end{lemma}

\begin{remark}
Under suitable assumptions, one may define $m-2$ geodesic curvature and (geodesic) torsion functions. For instance, the second geodesic torsion is defined as follows. Let $\mathcal{H}_{2} = (T \oplus N \oplus T_{g})^{\perp}$, let $\pi_{\mathcal{H}_{2}}$ be the orthogonal projection onto $\mathcal{H}_{2}$, and let
\begin{equation*}
T_{g,2} = -\pi_{\mathcal{H}_{2}} D_{t}T_{g}.
\end{equation*}
If $T_{g}$ is itself torsion-defining, i.e., there exists a smooth unit vector field $W_{2}$ along $\gamma$ that is everywhere parallel to $T_{g,2}$, then the function $\tau_{g,2} = \langle T_{g,2}, W_{2} \rangle$ is called the \textit{second geodesic torsion of $\gamma$ with respect to $N$}. Higher-order geodesic torsions are defined similarly.
\end{remark}

\section{Rotating the normal}\label{RotatingNormal}
Suppose that the normal vector $N$ along $\gamma$ rotates about the curve's tangent. Then how do the curvature vectors change? The purpose of this section is to answer such question.

Let $Z$ be a rotation of $N$. Then, by definition, there exists a unit normal vector field $H(N,Z) \equiv H \in \Gamma(\mathcal{H})$ along $\gamma$ such that $N$, $Z$, and $H$ are everywhere linearly dependent; besides, there is a continuous function $\theta \colon I \to \mathbb{R}$ such that
$$Z = -\sin(\theta)H + \cos(\theta)N.$$
Denoting $Z$ by $N(\theta)$, we call the function $\theta$ the \textit{rotation angle of $N(\theta)$ with respect to $H$}.

Now, let $(H_{1}, \dotsc, H_{m-2})$ be a smooth orthonormal frame for $\mathcal{H} = (E \oplus N)^{\perp}$, with $H_{1} = H$. It follows that
\begin{equation*}
N(\theta) = -\sin(\theta) H_{1} + \cos(\theta) N,
\end{equation*}
while the vector fields
\begin{align*}
H_{1}(\theta) &= \cos(\theta)H_{1} + \sin(\theta) N,\\
H_{2}(\theta) &= H_{2},\\
&\mathrel{\makebox[\widthof{=}]{\vdots}} \\
H_{m-2}(\theta) &= H_{m-2}
\end{align*}
span $\mathcal{H}(N(\theta)) = (E \oplus N(\theta))^{\perp}$.

\begin{lemma}\label{FrameRotationLM}
The curvature vectors of $\gamma$ with respect to $N(\theta)$ are given by
\begin{align*}
		K_{g}(\theta) &= \mleft(\kappa_{g}^{1} c + \kappa_{n} s \mright) H_{1}(\theta) + \kappa_{g}^{2} H_{2}(\theta) + \dotsb + \kappa_{g}^{m-2} H_{m-2}(\theta),\\
		K_{n}(\theta) &= \mleft( -\kappa_{g}^{1} s + \kappa_{n} c \mright) N(\theta),\\
		T_{g}(\theta) &= \mleft( \theta' + \tau_{g}^{1} \mright) H_{1}(\theta) + \mleft(\tau_{g}^{2} c - \mu_{2} s \mright) H_{2}(\theta) + \dotsb + \mleft( \tau_{g}^{m-2} c - \mu_{m-2} s \mright) H_{m-2}(\theta),
\end{align*}
where $\mu_{j} = \langle D_{t} H_{j}, H_{1} \rangle$, and where $c$ and $s$ are shorthands for $\cos(\theta)$ and $\sin(\theta)$, respectively.
\end{lemma}

\section{Proof of Theorem~\ref{mainTHM}}\label{ProofTheorem}

Here we prove our most general result, Theorem~\ref{mainTHM} in the introduction.

To begin with, suppose that $\gamma$ is a line of curvature of $S$ and that $N_{S}(\theta)$ is a parallel rotation of $N_{S}$. Then the geodesic torsion of $\gamma$ with respect to $N_{S}$ vanishes and the vector field $H(N_{S}, N_{S}(\theta))$ is parallel in $\mathcal{H}$.

Let $(H_{1}, \dotsc, H_{m-2})$ be a smooth orthonormal frame for $(E \oplus N_{S})^{\perp}$ such that $H_{1} = H$. Applying Lemma~\ref{FrameRotationLM}, we deduce that $T_{g}(\theta) = \theta' H_{1}(\theta)$, which implies that $N_{S}(\theta)$ is torsion-defining and that $\theta' = \pm \tau_{g}(\theta)$.

Since
\begin{equation*}
\int_{0}^{\ell} \theta' \, dt = \theta(\ell) - \theta(0),
\end{equation*}
it follows that, when $N_{S}(\theta)$ is a closed rotation of $N_{S}$,
\begin{equation*}
\int_{0}^{\ell} \tau_{g}(\theta) = 2n\pi \quad\text{for some $n \in \mathbb{Z}$},
\end{equation*}
as desired.

Conversely, given any (torsion-defining) unit normal vector field $N$ along $\gamma$, suppose that $W(N) \equiv W$ is parallel in $\mathcal{H}$. Choose an orthonormal frame $(H_{1}, \dotsc, H_{m-2})$ for $\mathcal{H}$, with $H_{1} = W$, and let 
\begin{align*}
	N(\theta) &= -\sin(\theta)H_{1} + \cos(\theta)N,\\  
	H_{1}(\theta) &= \cos(\theta)H_{1} + \sin(\theta)N,
\end{align*}
where
\begin{equation} \label{AngleFromGeodesicTorsionEQ}
\theta(t) = -\int_{0}^{t} \tau_{g}(s) \, ds.
\end{equation}
(Note that $N(\theta)$ is a parallel rotation of $N$.) 

Define a map $\sigma \colon [0, \ell] \times \mathbb{R}^{m-1} \to M$ by
\begin{equation*}
\sigma(t, u) = \exp_{\gamma(t)}\mleft(u^{1}H_{1}(\theta)(t) + u^{2}H_{2}(t) + \dotsb + u^{m-2} H_{m-2}(t)\mright).
\end{equation*}
It is clear that $\sigma$ is a smooth immersion in a neighborhood of $[0, \ell] \times \{0 \}$; besides, its image is normal to $N(\theta)$ along $\gamma$.

It remains to show that $\gamma$ is a line of curvature of $\sigma$, i.e., that the geodesic torsion $\tau_{g}(\theta)$ of $\gamma$ with respect to $N(\theta)$ vanishes. Differentiating \eqref{AngleFromGeodesicTorsionEQ}, we have
\begin{equation*}
\theta' = -\tau_{g}^{1},
\end{equation*}
which implies $\tau_{g}^{1}(\theta) = 0$, as desired. Since $\tau_{g}^{2} = \dotsb = \tau_{g}^{m-1} = 0$ and $H_{1}$ is parallel in $\mathcal{H}$, we conclude that $\tau_{g}(\theta) = 0$ by Lemma~\ref{FrameRotationLM}.

\section{Three-dimensional curves}\label{3DCurves}

Let $\gamma \colon I \to M$ be a Frenet curve, let $H_{1} = W(P)$, and let $(H_{2}, \dotsc, H_{m-2})$ be a parallel frame for the orthogonal complement of $H_{1}$ in $\mathcal{H}(P)$.

\begin{definition}\label{3DcurveDEF}
We say that $\gamma$ is \textit{three-dimensional} if the following equations hold:
\begin{align*}
D_{t}E &= \kappa P,\\
D_{t}H_{1} &= \tau P,\\
D_{t}H_{2} &= \dotsb = D_{t}H_{m-2} =0,\\
D_{t}P &= -\kappa E -\tau H_{1}.
\end{align*}
\end{definition}
It is clear that $\gamma$ is three-dimensional if and only if $W(P)$ is parallel in $\mathcal{H}(P)$.

The purpose of this section is to prove Theorem~\ref{mainCOR} in the introduction.

\begin{proof}[Proof of Theorem~\textup{\ref{mainCOR}}]
Suppose that $P$ is a parallel rotation of $N_{S}$, and let $\theta$ be the rotation angle of $P$ with respect to $W(P)$. We know from the proof of Theorem~\ref{mainTHM} that if $\gamma$ is a line of curvature and $P$ is a \emph{closed} rotation, then 
\begin{equation}\label{subintegralEQ}
\pm\int_{0}^{\ell} \tau \, dt =  \theta(\ell) - \theta(0) = 2n\pi \quad \text{for some $n \in \mathbb{Z}.$}
\end{equation}

On the other hand, applying Lemma~\ref{FrameRotationLM}, we observe that the normal curvature of $\gamma$ with respect to $N_{S}$ is related to the curvature $\kappa$ by the relation
\begin{equation*}
\kappa_{n} = \kappa(\theta) = \kappa \cos(\theta).
\end{equation*}

Suppose that $M$ is convex, so that $\kappa_{n} > 0$. Since $\kappa >0$, we have $\cos(\theta) > 0$, from which we conclude that
\begin{equation*}
\theta(\ell) - \theta(0) \in (-\pi, \pi).
\end{equation*}
Together with \eqref{subintegralEQ}, this implies $n =0$, as desired.
\end{proof}

\bibliographystyle{amsplain}
\bibliography{biblio}
\end{document}